\documentclass[12pt]{amsart}
\usepackage[colorlinks=true, pdfstartview=FitV, linkcolor=blue, citecolor=blue, urlcolor=blue, breaklinks=true]{hyperref}
\usepackage{amsmath,amsfonts,amssymb,amsthm,amscd,array,comment,
euscript,mathtools,etoolbox,latexsym,stmaryrd,mathrsfs,mathdesign,mathrsfs}
\usepackage{paralist}
\usepackage[usenames]{color}
\usepackage[all]{xy}

\DeclareMathAlphabet{\mathpzc}{OT1}{pzc}{m}{it}

\newtoggle{comments}
\newtoggle{details}
\newtoggle{detailsnote}

\toggletrue{comments}      

\toggletrue{details}       

\toggletrue{detailsnote}   

\newcommand\C{\mathbb{C}}
\newcommand\Z{\mathbb{Z}}

\newcommand\N{\mathbb{N}}

\newcommand\ch{\mathrm{ch}}

\newcommand\g{\mathfrak{g}}
\newcommand\fm{\mathfrak{m}}

\newcommand\fq{\mathfrak{q}}

\newcommand\fh{\mathfrak{h}}
\newcommand\fn{\mathfrak{n}}
\newcommand\fa{\mathfrak{a}}

\newcommand\fsl{\mathfrak{sl}}

\newcommand\fgl{\mathfrak{gl}}

\newcommand\fu{\mathfrak{u}}
\newcommand\fk{\mathfrak{k}}

\newcommand\hg{\hat{\mathfrak{g}}}

\newcommand\hh{\hat{\mathfrak{h}}}
\newcommand\hm{\hat{\mathfrak{m}}}
\newcommand\hk{\hat{\mathfrak{k}}}

\newcommand\hb{\hat{\mathfrak{b}}}

\newcommand\hs{\hat{\mathfrak{s}}}
\newcommand\hr{\hat{\mathfrak{r}}}

\newcommand\D{\Delta}

\newcommand\cH{\mathcal{H}}
\newcommand\tcH{\widetilde{\mathcal{H}}}

\newcommand\cS{\mathcal{S}}

\newcommand\cW{\mathcal{W}}

\newcommand\ba{\mathbf{a}}
\newcommand\bb{\mathbf{b}}
\newcommand\bc{\mathbf{c}}
\newcommand\bd{\mathbf{d}}
\newcommand\bg{\mathbf{g}}
\newcommand\bk{\mathbf{k}}

\newcommand\bi{\mathbf{i}}

\newcommand\bp{\mathbf{p}}

\newcommand\bU{\mathbf{U}}

\newcommand\bm{\mathbf{m}}

\DeclareMathOperator{\Span}{Span}
\DeclareMathOperator{\LinSpan}{LinSpan}

\DeclareMathOperator{\Id}{Id}

\DeclareMathOperator{\Ind}{Ind}

\DeclareMathOperator{\ad}{ad}

\DeclareMathOperator{\tr}{tr}

\theoremstyle{plain}
\newtheorem{theo}{Theorem}[section]
\newtheorem*{theo*}{Theorem}
\newtheorem{prop}[theo]{Proposition}
\newtheorem{lem}[theo]{Lemma}
\newtheorem{cor}[theo]{Corollary}

\theoremstyle{definition}

\newtheorem*{rem*}{Remark}
\newtheorem{rem}[theo]{Remark}

\newtheorem{example}[theo]{Example}

\numberwithin{equation}{section}

\allowdisplaybreaks

\iftoggle{comments}{
\newcommand{\comments}[1]{ \begin{center} \color{blue} \parbox{5 in}{{\bf {\footnotesize Comments:  }}{\footnotesize \textit{#1}}} \end{center}}
}{
\newcommand{\comments}[1]{}
}

\iftoggle{details}{
\newcommand{\details}[1]{\smallskip \color{blue} \begin{footnotesize} \textbf{Details:} #1 \end{footnotesize} \color{black}}
}{
\newcommand{\details}[1]{}
}

\begin{document}
%

\title[Non-standard Verma type modules for $\fq(n)^{(2)}$]
{Non-standard Verma type \\ 
modules for $\fq(n)^{(2)}$}

\title{Non-standard Verma type modules for $\fq(n)^{(2)}$}
\author{Lucas Calixto}
\address{Department of Mathematics\\
		   Federal University of Minas Gerais\\
		   Belo Horizonte \\
		    Brazil}
		   \email{lhcalixto@ufmg.br}	
\author{Vyacheslav Futorny}
\address{Institute of Mathematics and Statistics \\
			 University of S\~ao Paulo \\
			 S\~ao Paulo \\
			 Brazil}
\email{futorny@ime.usp.br}
\subjclass{Primary 17B67}
\date{}




\maketitle

\begin{abstract}
We study non-standard Verma type modules over the Kac-Moody queer Lie superalgebra $\mathfrak{q}(n)^{(2)}$. We give a sufficient condition under which such modules  are irreducible. We also give a classification of all irreducible diagonal $\mathbb{Z}$-graded modules over certain Heisenberg Lie superalgebras contained in $\mathfrak{q}(n)^{(2)}$.
\end{abstract}

\section*{Introduction}
Kac-Moody algebras and their representations  play a very important role in many areas of mathematics and  physics. The "super" version of these algebras was introduced in \cite{Kac77}. Affine Kac-Moody superalgebras are those of finite growth. Affine symmetrizable superalgebras were described in \cite{Ser11} and \cite{Van89}. Theory of Verma type modules for  affine Lie superalgebras  was developed in \cite{EF09}  and 
\cite{CF18}. In particular, given a Borel subsuperalgebra $\hb$ of the affine Lie superalgebra  $\hg$ and a $1$-dimensional representation $\C_{\lambda}$ of $\hb$ for some weight $\lambda$ of the Cartan subalgebra of $\hg$, 
one can construct the \emph{Verma type module}
	\[
M_{\hb}(\lambda) := \Ind^{\hg}_{\hb} \C_{\lambda}.
	\]
This module  admits a unique maximal proper submodule, and thus, a unique simple quotient.  The Verma type module  is \emph{non-standard} if 
$\hb$ does not contain all positive root subspaces for some basis of the root system of $\hg$.  In the case the finite-dimensional Lie superalgebra  associated to $\hg$ is a contragredient Lie superalgebra, all Borel subsuperalgebras of $\hg$ were described in  \cite{CF18}, see also \cite{DFG09}. The paper  \cite{CF18} also gives a criterion for the irreducibility for non-standard Verma type module. 

 Non-symmetrizable affine Lie superalgebras were classified in \cite{HS07}. In particular, this classification includes a degenerate family of affine Lie superalgebras, series $\fq(n)^{(2)}$. These superalgebras  are twisted affinizations of queer Lie superalgebras $\fq(n)$.  Structure of Verma modules (= standard Verma type modules) over the twisted affine superalgebra $\fq(n)^{(2)}$ with $n\geq 3$ was studied in \cite{GS08}. The current paper advances the theory of Verma type modules for the affine queer Lie superalgebra. We establish sufficient conditions for the irreducibility of all non-standard Verma type modules (Theorem \ref{thm-extmain} and Theorem \ref{thm-main}).  We also consider  modules induced from analogs of Heisenberg subsuperalgebra and give a criterion of their irreducibility (Theorem \ref{thm-Heis}, Corollary \ref{cor:X=empty-irr.criterion}).

\medskip

\noindent {\textbf{Notation}} The ground field is $\C$. All vector spaces, algebras, and tensor products are considered to be over $\C$, unless otherwise stated. For a vector space $V$ we denote by $\Lambda (V)$ its Grassmann algebra (i.e., its exterior algebra). For any Lie superalgebra $\fa$ we let $\bU(\fa)$ denote its universal enveloping algebra. 

\renewcommand\thesubsection{\arabic{subsection}}

\section{Preliminaries}

Let $\fq=\fq(n)$ for $n\geq 3$, be the queer Lie superalgebra, that is, 
	\[
\fq := \lbrace \left(\begin{array}{c|c}
   A & B  \\
    \hline
    B & A
  \end{array}\right)\mid A\in \fgl(n+1), B\in \fsl(n+1) \rbrace.
	\]
Let $\fq_0$ and $\fq_1$ be the even and odd parts of $\fq$, respectively. Choose a Cartan subalgebra $\fh_\fq = \fh_{0}\oplus \fh_1$ of $\fq$ (i.e. $\fh_0$ a Cartan subalgebra of $\fq_0$) and let $\fq=\fh_\fq\oplus\left(\bigoplus_{\alpha\in \dot{\D}} \fq^\alpha\right)$ be the root space decomposition of $\fq$, where $\fq^\alpha$ denotes the root space associated to the root $\alpha\in \dot{\D}\subseteq \fh_0^*$. Recall that every root of $\dot{\D}$ is both even and odd, meaning that, for any $\alpha\in \dot{\D}$, $\fq^\alpha\cap \fq_i\neq 0$, for $i=0,1$. Recall also that $\dot{\D}=\dot{\D}_0=\dot{\D}_1 = \{\varepsilon_i-\varepsilon_j\mid i\neq j\}$. 

Although neither $\fq$ nor its affinization $\fq^{(1)}$ are Kac-Moody Lie superalgebras, i.e. admit a set of simple generators, after a twist of $\fq^{(1)}$ by an involution we obtain a regular quasisimple  Kac-Moody superalgebra $\hg:=\fq^{(2)}$ (see \cite{Ser11}). As a super vector space we have that
	\[
\hg = L(\fsl(n))\oplus \C K\oplus \C D,\quad \hg_0 = \fsl(n)\otimes \C [t^{\pm2}]\oplus \C K\oplus \C D, \quad \text{and}\quad \hg_1 = \fsl(n)\otimes t\C [t^{\pm2}],
	\]
where for any Lie superalgebra $\fk$, $L(\fk) := \fk\otimes \C [t^1, t^{-1}]$ is its associated loop superalgebra, $K$ is a central element, and, for all $x(k):=x\otimes t^k\in L(\fk)$ with $x\in \fk$ and $k\in \Z$, we have $[D, x(k)]=kx(k)$. Let $\g = \fsl(n)$. Then for any $x, y\in \g$, the bracket of $\hg$ is given as follows:
	\[
[x(k), y(m)] = (xy-yx)(k+m),
	\]
if $km$ is even; and if we define $\iota:\fgl(n)\to \fsl(n)$ by $x\mapsto x-\frac{\tr(x)}{n} I_n$ where $I_n$ is the $n\times n$ identity matrix, then
	\[
[x(k), y(m)] = \iota(xy+yx)(k+m) + 2\delta_{-k, m}\tr(xy)K,
	\]
if $km$ is odd. Notice that $K$ does not lie in $[\hg_0, \hg]$, but it lies in $[\hg_1, \hg_1]$. For convenience we set 
	\[
(xy-yx) := [x,y]_0,\text{ and } \iota(xy+yx) := [x,y]_1.
	\]
Hence, in this notation we have that
	\[
[x(k), y(m)] = [x,y]_0(k+m),\text{ and } [x(k), y(m)] = [x,y]_1(k+m) + 2\delta_{-k, m}\tr(xy)K,
	\]
if $km$ is even/odd, respectively. 

\begin{rem}
Notice that if we assume $m\in 2\Z$, then the bracket between any two elements $x(m), y(k)\in L(\g)$ reduces to the bracket in the loop Lie algebra $L(\g)$.
\end{rem}

Fix a Cartan subalgebra of $\hg$
	\[
\hh := \fh\otimes 1 \oplus \C K\oplus \C D
	\]
where $\fh$ is the Cartan subalgebra of diagonal matrices in $\g$, and for each $\alpha\in \dot{\D}$, choose $f_\alpha\in \g^{-\alpha}$, $e_\alpha\in \g^\alpha$ and $h_\alpha\in \fh$ such that $[f_\alpha, e_\alpha]_0 = h_\alpha$. 

Notice that, for $g_{\varepsilon_i - \varepsilon_j}\in \g^{\varepsilon_i - \varepsilon_j}$, we have
	\[
[h, g_{\varepsilon_i - \varepsilon_j}]_1 = (\varepsilon_i + \varepsilon_j)(h)g_{\varepsilon_i - \varepsilon_j},\text{ for all } h.
	\]
For simplicity, if $\alpha = \varepsilon_i - \varepsilon_j$, then we set ${\bar \alpha} := \varepsilon_i + \varepsilon_j$. Thus, in this notation, we have that
	\[
[h, g_\alpha]_1 = {\bar \alpha}(h)g_\alpha,\text{ for all } h\in \fh.
	\]
Moreover, if $\alpha_i\neq -\alpha_j$, then
	\[
[e_{\alpha_i}, f_{\alpha_j}]_1 = g_{\alpha_i + \alpha_j},
	\]
where $g_{\alpha_i+\alpha_j} = 0$ if $\alpha_i+\alpha_j\notin \dot{\D}$, $g_{\alpha_i+\alpha_j} = f_{\alpha_i+\alpha_j}$ if $\alpha_i+\alpha_j\in \dot{\D}^-$ and $g_{\alpha_i+\alpha_j} = e_{\alpha_i+\alpha_j}$ if $\alpha_i+\alpha_j\in \dot{\D}^+$. Finally, for $\alpha = \varepsilon_i - \varepsilon_j$ we have
	\[
[e_{\alpha}, f_{\alpha}]_1 = \iota(h_{\alpha}'),
	\]
where $h_{\alpha}' = E_{i,i} + E_{j,j}$.

If we identify $K$ with $(1/n) I_n$, then $\fh\otimes 1\oplus \C K$ can be identified with the Cartan subalgebra of diagonal matrices of $\fgl(n)$. Let $H_1,\ldots, H_n$ denote the standard basis of it (i.e. $H_i = E_{ii}$). The root system of $\hg$ with respect to $\hh$ is given by $\D = \{\alpha + k\delta,\ m\delta\mid \alpha\in \dot{\D},\ k\in \Z,\ m\in \Z\setminus\{0\}\}$. Moreover, $p(\alpha + k\delta) = p(k)$ and $p(m\delta) = p(m)$, where $p(k)$ denotes the parity of $k$, and by abuse of notation we are denoting the parity of a root $\beta$ also by $p(\beta)$. Finally, for a subalgebra $\fa\subseteq \hg$ we set 
	\[
\D(\fa) := \{\alpha\in \D\mid \hg_\alpha\subseteq \fa\}.
	\]

Consider the subalgebra $\widehat{\cH} = \widehat{\cH}_0\oplus \widehat{\cH}_1$ generated by the imaginary root spaces of $\hg$. Then
	\[
\widehat{\cH}_0 = \sum_{r\in \Z} \fh\otimes t^{2r}\oplus \C K,\quad \widehat{\cH}_1 = \sum_{r\in \Z} \fh\otimes t^{2r+1}.
	\]
Notice that the center of $\widehat{\cH}$ equals to $\widehat{\cH}_0$, the odd part $\widehat{\cH}_1$ is spanned by $\{(H_i - H_{i+1})(2r+1)\mid r\in \Z\}$ and the relations in $\widehat{\cH}_1$ are given by
	\[
[x(2r+1), y(-2r-1)] = 2xy,\quad  [x(2r+1), y(2s+1)] = \iota(2xy)(2(r+s+1))
	\]
for $r+s+1\neq 0$. In particular, differently from the case of basic classical Lie superalgebras, the subalgebra $\widehat{\cH}$ is not isomorphic to a Heisenberg algebra.

\section{Generalized Verma type modules}

Since the root system $\D$ of $\hg$ is the same as that of $\widehat{\fsl}(n)$, the sets of positive roots of $\D$ are obtained in the same way: fix $\Pi\subseteq \dot{\D}$ a set of simple roots, pick a subset $X\subseteq \Pi$, and let $\cW$ denote the Weyl group of $\widehat{\fsl}(n)$. Let $\dot{\D}^+ = \langle \Pi\rangle_{\Z_{>0}}\cap \dot{\D}$, $\dot{\D}(X)^+ = \langle X\rangle_{\Z_{>0}}\cap \dot{\D}$, and $\dot{\D}(X) = \langle X\rangle_{\Z}\cap \dot{\D}$.  Associated to $X$ we define
\begin{align*}
\D(X)^+ & := \{\alpha+k\delta\mid \alpha\in \dot{\D}^+\setminus \dot{\D}(X)^+,\ k\in \Z\} \\ & \cup \{\alpha+k\delta\mid \alpha\in \dot{\D}(X)\cup \{0\},\ k\in \Z_{>0}\}\cup \dot{\D}(X)^+.
\end{align*}
Then $\D(X)^+$ is a set of positive roots of $\D$, and up to $\cW \times \{\pm 1\}$-conjugation, every set of positive roots is of this form for some set of simple roots $\Pi$ and some subset $X\subseteq \Pi$. 

Consider the following subalgebras associated to $X$:
\begin{enumerate}
\item []
	\[
\fm(X) := \fm(X)^-\oplus \fh\oplus \fm(X)^+, \quad \fm(X)^\pm :=\bigoplus_{\alpha\in \dot{\D}(X)^\pm} \g_\alpha.
	\]
\item []
	\[
\fu(X)^\pm :=\bigoplus_{\alpha\in \dot{\D}^\pm \setminus \dot{\D}(X)^\pm} \g_\alpha.
	\]
\end{enumerate}
Thus 
	\[
 \g = \fu(X)^-\oplus \fm(X)\oplus \fu(X)^+\quad \text{and}\quad \hg = L(\fu(X)^-)\oplus \hm(X)\oplus L(\fu(X)^+),
	\]
where $\hm(X) = L(\fm(X))\oplus \C K \oplus \C D$.

Consider now the subalgebra
	\[
\fk(X) := \fm(X)^-\oplus \fh_X\oplus \fm(X)^+,\text{ where } \fh_X := \bigoplus_{\alpha\in \dot{\D}(X)^+} [\g_{-\alpha}, \g_\alpha].
	\]
Then $\fm(X) = \fk(X)\oplus \fh^X$, where $\fh^X := \{h\in \fh\mid \alpha(h)=0,\ \forall \alpha\in \dot{\D}(X)\}$ is the center of $\fm(X)$. Set
	\[
\hk(X) := L(\fk(X))\oplus \C K\oplus \C D\oplus \fh^X
	\]
with standard triangular decomposition
	\[
\hk(X)= \hk(X)^-\oplus \hh\oplus \hk(X)^+,\quad \hk(X)^\pm = (\fk(X)\otimes t^{\pm 1}\C[t^{\pm 1}])\oplus \fm(X)^\pm.
	\]
In particular, we have that
	\[
\hm(X) = \fh^X\otimes t^{-1}\C[t^{-1}]\oplus (\hk(X)^- \oplus \hh \oplus  \hk(X)^+)\oplus \fh^X\otimes t \C[t]
	\]
and 
	\[
\hg = (L(\fu(X)^-)\oplus \fh^X\otimes t^{-1}\C[t^{-1}])\oplus \hk(X)^- \oplus \hh \oplus  \hk(X)^+\oplus (\fh^X\otimes t \C[t]\oplus L(\fu(X)^+)).
	\]

\begin{rem}
\begin{enumerate}
\item Differently from the case of basic classical Lie superalgebras (this includes all simple Lie algebras), the imaginary subalgebra 
	\[
\cH(X) := (\fh^X\otimes t^{-1}\C[t^{-1}])\oplus (\fh^X\oplus \C K) \oplus (\fh^X\otimes t\C[t])
	\] 
is not a Heisenberg algebra. Another difference (from the Lie algebra case) is that we do not have that $[\cH(X), \hk(X)]=0$. In fact, 
	\[
[\cH_{0}(X), \hk(X)]= [\cH(X), \hk(X)_0] = 0, \text{ but } [\cH_{1}(X), \hk(X)_1] \neq 0.
	\]
Compare also with the isotropic case of \cite{CF18}.
\item Heisenberg algebras admit a family of triangular decompositions parametrized by maps $\varphi:\N\to \{\pm\}^d$, where $d$ is a certain dimension. It is worth noting that the algebra $\cH(X)$ does not admit such decompositions, except the trivial ones (i.e. when $\varphi(i)=(+,\ldots, +)$ for all $i\in \N$, or $\varphi(i)=(-,\ldots, -)$ for all $i\in \N$).
\end{enumerate} 
\end{rem}

Consider the triangular decomposition of $\cH(X)$
	\[
\cH(X) = \cH(X)^-\oplus (\fh^X\oplus \C K) \oplus \cH(X)^+,
	\]
where 
	\[
\cH(X)^\pm := \fh^X\otimes t^{\pm1}\C[t^{\pm1}],
	\]
and define $\cH(X)_i^\pm := \cH(X)_{i}\cap \cH(X)^\pm$, for $i\in \Z_2$. Then we have a commutative algebra 
	\[
\cS(X) := \bU(\cH(X)_0^-),
	\]
and we let $\cS(X)^+$ denote the augmentation ideal of $\cS(X)$.

Consider the triangular decompositions
	\[
\hm(X) = \hm(X)^- \oplus \hh\oplus \hm(X)^+,\quad \text{where}\quad \hm(X)^\pm = \cH(X)^\pm \oplus \hk(X)^\pm
	\]
and 
	\[
\hg = \hg(X)^-\oplus \hh\oplus \hg(X)^+,\quad \text{where}\quad \hg(X)^\pm = L(\fu(X)^\pm)\oplus \hm(X)^\pm.
	\]

Fix the subalgebra $\hb(X) :=\hh\oplus \hg(X)^+$ of $\hg$. Notice that $\hg(X)^+\cap \hm(X) = \hm(X)^+$, $\hg(X)^+\cap \hk(X) = \hk(X)^+$, and $\hg(X)^+\cap \cH(X) = \cH(X)^+$. In what follows, we fix a set $X\subseteq \Pi$ and we drop the $X$ from the notation above (for instance, we write $\hm^+$ instead of writing $\hm(X)^+$).

Let $\lambda\in \hh^*$, $\hs\in \{\hg, \hm, \hk, \cH\}$, and $\hr = \hs\cap \hb$. Then we define the Verma $\hs$-module
	\[
M(\hs, \lambda) := \bU(\hs)\otimes_{\bU(\hr)} \C v_\lambda,
	\]
where $\C v_\lambda$ is the $\hr$-module whose action of $\hh$ is determined by $\lambda$ and the action of the nilpotent radical of $\hr$ is trivial. The unique irreducible quotient of $M(\hs, \lambda)$ will be denoted $L(\hs, \lambda)$. Also, for $\hs, \hr$ such that either $\hr = \hk$ and $\hs = \hm$, or $\hr = \hm$ and $\hs = \hg$, and an $\hr$-module $N$ we define the module
	\[
M(\hs, \hr; N) := \bU(\hs)\otimes_{\bU(\hr)} N,
	\]
where $\cH^+$ is assumed to act trivially on $N$ if $\hr = \hk$ and $\hs = \hm$, and $L(\fu^+)$ is assumed to act trivially on $N$ if $\hr = \hm$ and $\hs = \hg$.  Notice that 
	\[
M(\hg, \lambda)\cong M(\hg, \hm; M(\hm, \lambda))\text{ and } M(\hm, \lambda)\cong M(\hm, \hk; M(\hk, \lambda)).
	\]

Using the terminology of \cite{Fut97}, the module $M(\hg, \hm; N)$ is called a \emph{generalized Verma type module}, or a \emph{generalized Imaginary Verma module}. When $N$ is an irreducible weight $\hr$-module, $M(\hs, \hr; N)$ admits a unique irreducible quotient which will be denoted by $L(\hs, \hr; N)$.

\section{Irreducible $\cH$-modules} 
Consider the triangular decomposition
	\[
\cH = \cH^-\oplus \C K\oplus \cH^+. 
	\]
Then we have the following character formula 
	\[
\ch \ M(\cH, \lambda) = e^\lambda \prod_{\alpha\in \D(\cH^-)_0} (1-e^{-\alpha})^{-1} \prod_{\alpha\in \D(\cH^-)_1}(1+e^{-\alpha}).
	\]

Notice that the subalgebra $\cS$ lies in the center of $\bU(\cH)$ and acts freely on $M(\cH, \lambda)$. Then any ideal $J$ of $\cS$ defines the $\cH$-submodule $JM(\cH, \lambda)$ of $M(\cH, \lambda)$. On the other direction, for any $\cH$-submodule $N\subseteq M(\cH, \lambda)$ we define an ideal $J_N$ of $\cS$ by requiring the equality:
	\[
N\cap \cS v_\lambda = J_N v_\lambda.
	\]
In other words, $J_N = \{a\in \cS\mid av_\lambda \in N\}$.

Let $D_\delta^X$ be the matrix determined by the pairing 
	\[
(\fh^X\otimes \C t)\times (\fh^X\otimes \C t^{-1})\to \hh,\quad (x,y) \mapsto [x,y],
	\]
and consider $\det D_\delta^X$ as an element of the symmetric algebra $S(\hh)$.

\begin{example}
If $n=3$ and $X = \{\varepsilon_1-\varepsilon_2\}$, then $\fh^X = \C h^1$, where $h^1 = H_1 + H_2 -2 H_3$. In particular, $\det D_\delta^X = 2(H_1 + H_2 - 2 H_3)^2$. If $X = \emptyset$ and $n\geq 3$, then $\fh^X = \fh$, and $\det D_\delta^X = 2^{n-1}H_1\cdots H_n(\frac{1}{H_1}+\cdots+\frac{1}{H_n})$ (see \cite{GS08}).
\end{example}

\begin{prop}\label{prop:ideals.S.bij.subm.M(H)} 
The $\cH$-module $M(\cH, \lambda)$ is reducible. If $\det D_\delta^X(\lambda)\neq 0$, then there is a bijection between submodules of $M(\cH, \lambda)$ and ideals of $\cS$. In particular,  $L(\cH, \lambda)\cong \Lambda(\cH_{1}^-)$ as vector spaces.
\end{prop}
\begin{proof}
The fact that $\cH_0^-$ is in the center of $\cH$ implies that any ideal $J$ of $\cS$ defines a submodule of $M(\cH, \lambda)$, namely, $JM(\cH, \lambda)$. Thus the first statement follows.

Now, if we assume $\det D_\delta^X(\lambda)\neq 0$, then we can use similar arguments to those of \cite[Proposition~3]{GS08} to prove that there is a bijection between ideals of $\cS$ and submodules of $M(\cH, \lambda)$. Namely, let $M = M(\cH, \lambda)$, and let $N$ be a submodule of $M$. We claim that $N = J_NM$. Indeed, let $jm\in J_NM$. Then, writing $m = u v_\lambda$ with $u\in \bU(\cH^-)$, we get that
	\[
jm=juv_\lambda=ujv_\lambda\in uN\subseteq N.	
	\]
Thus $J_NM\subseteq N$. In order to prove the other inclusion, we consider the canonical projection $\pi: M\to V:=M/J_NM$, $W = \pi(N)$, and $R = \cS/J_N$.  Notice that $V$ is free as an $R$-module, and that $W\cap Rv_\lambda=\pi(N\cap \cS v_\lambda)=\pi(J_N v_\lambda)=0$. Now we suppose that $W\neq 0$ to get a contradiction. 

Let $h^1,\ldots, h^t$ be any fixed basis of $\fh^X$, and set $X_{i,m} := h^i(m)$, and $Y_{i,m} := h^i(-m)$. Recall from the commutation relations of $\cH$ that $[X_{i,m}, Y_{k,m}] = [X_{i,0}, Y_{k,0}]$, and since we are assuming that $\det D_\delta^X(\lambda)\neq 0$, we may consider that the basis elements $h^1,\ldots, h^t$ were chosen so that $\lambda([X_{i,j}, Y_{k,j}]) = \delta_{i,k}$. Notice that the elements $X_{i,j}$ for $i=1,\ldots, r$ and $m\geq 0$ form a basis for $\cH_{1}^-$. In particular, if we let $X_{i,m}\geq X_{k,n}$ if $m\geq n$ or $m=n$ and $i\geq k$, then the monomials $X_{i_1,m_1}\cdots X_{i_s, m_s}$ with $X_{i_1,m_1}>\cdots > X_{i_s,m_s}$ form a basis $B$ of $V$ over $R$.

Since we are assuming $W\neq 0$, and since $W\cap Rv_\lambda=0$, we can choose a nonzero $v\in W$ such that the maximal $X_{i,m}$ that occurs in the expression of $v$ as a linear combination of elements  of $B$ is minimal among all nonzero vectors of $W$. Now we write $v = X_{i,m}w + u$ for nonzero $w$ and $u$ such that all factors occurring in $w$ and $u$ are less than $X_{i,m}$. Thus 
	\[
Y_{i,m}u = Y_{i,m}w = 0\text{ and } Y_{i,m}v = w,
	\]
as $[X_{i,m}, Y_{i,m}]$ is in the center of $\cH$ and it acts as $\lambda([X_{i,m}, Y_{k,m}]) = 1$ on $v_\lambda$. But this implies $0\neq w\in W$, and all factors occurring in $w$ are less than $X_{i,m}$, which is a contradiction.
\end{proof}

\begin{cor}
Suppose that $\det D_\delta^X(\lambda)\neq 0$.  Then we have the character formula 
	\[
\ch \ L(\cH, \lambda) = \ch \ \Lambda(\cH_{1}^-) = e^\lambda \prod_{\alpha\in \D(\cH^-)_1}(1+e^{-\alpha}).
	\]
\end{cor}
\begin{proof}
This follows from the isomorphism of vector spaces $L(\cH, \lambda)\cong \Lambda(\cH_{1}^-)$.
\end{proof}

\subsection{Modules for Heisenberg Lie superalgebra}
In this section we consider the special case where $X=\emptyset$, and, in particular, $\fh^X = \fh$ and $\cH=\cH(X)=L(\fh)\oplus \C K$. 

Define  
	\[
\cH_0':= \bigoplus_{r\in \Z}\fh\otimes t^{2r}.
	\]
It is clear that $\cH_0'$ is an ideal of $\cH$, and $K\notin \cH_0'$. Define
	\[
\tcH:=\cH/\cH_0'.
	\]

\begin{lem}\label{lem:good.basis}
Let $\pi:\cH\to \tcH$ be the canonical projection. Then there exists a basis $\{h^1,\ldots, h^{n-1}\}$ of $\fh$ such that $\pi(h^ih^j) = \delta_{ij}K$.
\end{lem}
\begin{proof}
The set $\{H_1+\cdots +H_{i} - iH_{i+1}\mid 1\leq i\leq n-1\}$ is a basis of $\fh$ such that $H^iH^j\in \fh$ if and only if $i\neq j$. Then a suitable normalization of this basis gives the required one.
\end{proof}

Now we have the following result:

\begin{prop}\label{prop-Heis}	
 $\tcH$ is an infinite dimensional Heisenberg Lie superalgebra such that 
 	\[
\tcH \cong \C K \oplus  \bigoplus_{r\in \Z}\fh\otimes t^{2r+1}
 	\]
as vector spaces, where  $[h\otimes t^{2r+1}, h'\otimes t^{-2r-1}]$ is a multiple of $K$ and $[h\otimes t^{2r+1}, h'\otimes t^{2s+1}]=0$ for all $h, h'\in \fh$ and  all integer $r, s$ with $r+s+1\neq 0$. 
\end{prop}

Fix a basis $h^1, \ldots, h^{n-1}$ of $\fh$ as in Lemma~\ref{lem:good.basis}, and let $\varphi : \N\to \{\pm\}^{n-1}$ be a map of sets. Then $\varphi$ induces a triangular decomposition on $\tcH$:
	\[
\tcH = \tcH_\varphi^-\oplus \C K\oplus \tcH_\varphi^+,
	\]
where 
	\[
\tcH_\varphi^\pm = \left( \bigoplus_{n\in \N,\ 1\leq i\leq t, \ \varphi(n)_i=\pm} \C h^i\otimes t^{2n+1}\right) \oplus \left( \bigoplus_{m\in \N,\ 1\leq i\leq t, \  \varphi(m)_i=\mp} \C h^i\otimes t^{-(2m+1)}\right),
	\]
and $\varphi(n)=(\varphi(n)_1, \ldots, \varphi(n)_{n-1})$.  The Verma module  associated to such a decomposition is called the \emph{$\varphi$-Verma module} and it is denoted by $M_\varphi(\tcH, a)$, where $a\in \C$ is the value of $K$ on $M_\varphi(\tcH, a)$. The module $M_\varphi(\tcH, a)$ is isomorphic (as a vector space) to $\bU( \tcH_\varphi^-)$ which is nothing but the Grassmann algebra $\Lambda(\tcH_\varphi^-)$. Finally let $L_\varphi(\tcH, a)$ denote the unique irreducible quotient of $M_\varphi(\tcH, a)$. 

\begin{rem}
Notice that every $\tcH$-module can (and will) be regarded as an $\cH$-module via the canonical projection $\cH \twoheadrightarrow  \tcH$.
\end{rem}

\begin{cor}
If $\lambda(\fh)=0$, then the action of $\cH$ on $L(\cH, \lambda)$ factors through the epimorphism
	\[
\cH\twoheadrightarrow \tcH.
	\]
In particular, if $\lambda(K):=a\neq 0$, then $L(\cH, \lambda)\cong M_\varphi(\tcH, a)$ as $\cH$-modules, where $\varphi(i) = (+,\ldots, +)$ for all $i\in \N$ (i.e. $M_\varphi(\tcH, a)$ is nothing but the standard Verma module of $\tcH$).
\end{cor}
\begin{proof}
We have $(\fh\otimes t^{2}\C[t]) L(\cH, \lambda)=0$, since $\fh\otimes t^{2}\C[t]$ is in the center of $\cH$ and it acts trivially on $v_\lambda$. Next, $\fh\otimes t^{-2}\C[t^{-1}]$ is contained in the maximal ideal $\cS^+$ of $\cS$, and then, by Proposition~\ref{prop:ideals.S.bij.subm.M(H)}, we must have $(\fh\otimes t^{-2}\C[t^{-1}]) L(\cH, \lambda)=0$. Finally, since $\lambda(\fh)=0$, we conclude that $\cH_0' L(\cH, \lambda)=0$ and the first statement follows.

Using similar arguments as those of \cite[Proposition~3.3]{BBFK13} one easily shows that $M_\varphi (\tcH, a)$ is an irreducible $\tcH$-module if and only if $a\neq 0$. Thus the result follows.
\end{proof}

Let $N$ be an irreducible $\cH$-module such that $\fh N=0$.  We are interested in the case when $N$ is $\Z$-graded. Then we can define the action of $D$ on $N$ by $D |_{N_i} = i\Id$. Notice that under such conditions $\cH_0'$ must act trivially on $N$ (indeed, $N$ is irreducible and $\Z$-graded, $\cH_0'$ is central in $\cH$, $\fh N=0$ and any element of $\fh\otimes t^{2r}$ with $r\in \Z^\times$ have degree different from $0$).

Set $x_k^j=h^j\otimes t^{k}$, $k\in \Z$, $j=1, \ldots, n-1$, so that 
	\[
\tcH  \cong \C K \oplus  \bigoplus_{r\in \Z,\  j=1, \ldots, t}\C x_{2r+1}^j
	\] 
and $[x_{2r+1}^j, x_{2s-1}^i]=\delta_{ij}\delta_{r, -s}K$, after suitable rescaling (see Lemma~\ref{lem:good.basis} and Proposition~\ref{prop-Heis}). Also set 
	\[
d_{2r+1}^j := x_{-2r-1}^j x_{2r+1}^j,\quad r\in \Z_{\geq 0},\  j=1, \ldots, n-1.  
	\]

Since $K$ is central and $N$ is irreducible, we have that $K$ acts on $N=\sum_{i\in \Z}N_i$ via multiplication by some $a\in \C$. Assume that $a\neq 0$, and fix a nonzero $v\in N_i$ for some $i$. Then 
	\[
(d_{2r+1}^j)^2v = (x_{-2r-1}^j x_{2r+1}^j)(x_{-2r-1}^j x_{2r+1}^j) = x_{-2r-1}^j(K - x_{-2r-1}^j x_{2r+1}^j)x_{2r+1}^j =  a d_{2r+1}^j v,
	\]
that is, $d_{2r+1}^j$ is diagonalizable on $N_i$ and has eigenvalues $a$ or $0$. Now we have:

\begin{lem}\label{lem-eigen}
If  $d_{2r+1}^j v=av$, then $x_{-2r-1}^j v=0$. On the other hand, if $d_{2r+1}^j v=0$, then $x_{2r+1}^j v=0$.
\end{lem}

\begin{proof}
The fact that $x^j_{-2r-1}d^j_{2r+1}=0$ implies the first statement. For the second statement observe that $d_{2r+1}^j v=0$ implies $x_{2r+1}^jx_{-2r-1}^j v = av$. Hence the result follows.
\end{proof}

A non-zero $\Z$-graded $\cH$-module $N$ is \emph{diagonal} if all $d_{2r+1}^j$ are  simultaneously diagonalizable  for $r\in \Z_{\geq 0}$, $j=1, \ldots, n-1$. Let $N_i$ be a graded component of a diagonal $\Z$-graded $\cH$-module $N$. We associate to  $N_i$ a $t$-tuple $(\mu^1, \ldots, \mu^{n-1})$ of infinite sequences  $\mu^j=(\mu_{2r+1}^j)$ consisting of the eigenvalues $\mu_{2r+1}^j$ of $d_{2r+1}^j$, $r\in \Z_{\geq 0}$, $j=1, \ldots, n-1$. In what follows we classify all diagonal irreducible modules with trivial action of $\fh$, and we describe their structure. 

\begin{theo}\label{thm-Heis}
Let $N$ be an irreducible diagonal $\Z$-graded $\cH$-module, such that $\fh N=0$ and $Kv=av$ for some $a\in \C$ and all $v\in N$. Then the following hold:
\begin{enumerate}
\item $\cH_0'$ acts trivially on $N$, which is irreducible $\tcH$-module;
\item If $v\in N$ is a nonzero homogeneous element, then $v$ is $\varphi_{\mu} $-highest vector, where $\varphi_{\mu} $ is determined by the eigenvalues of $d_{2r+1}^j $ on $v$, and $N\simeq L_{\varphi_{\mu}}(\tcH, a)$ up to a shift of gradation. In particular, if  $a\neq 0$, then $N\simeq M_{\varphi_{\mu}}(\tcH, a)$ up to a shift of gradation;
\item If $a=0$, then $N$ is the trivial $1$-dimensional module. 
\item If $a\neq 0$, then $M_{\varphi_{\mu}}(\tcH, a)$ has finite dimensional graded components if and only if $\varphi_{\mu}$ differs  from $\varphi_{\nu}$ only in finitely many places, where $\nu_{2k+1}^j=0$ for all $k\in \Z_{\geq 0}$, $j=1, \ldots, n-1$, or $\nu_{2k+1}^j\neq 0$ for all $k\in \Z_{\geq 0}$, $j=1, \ldots, n-1$.
\end{enumerate}
\end{theo}

\begin{proof} Part(a): this follows from the fact that $N$ is irreducible and $\Z$-graded, $\cH_0'$ is central and its elements have degree different from $0$. Part(b): let $N_i\neq 0$ such that all $d_{2r+1}^j $ are simultaneously diagonalizable with eigenvalues $\mu_{2r+1}^j$. Set $\mu^j=(\mu_{2r+1}^j)$, $r\in \Z_{\geq 0}$, $j=1, \ldots, n-1$. By Lemma \ref{lem-eigen}, each
  $(\mu^1, \ldots, \mu^{n-1})$ defines a function $\varphi_{\mu}: \N\to \{\pm\}^{n-1}$, where $\varphi_{\mu}(k)_j=+$ if $\mu_{2k+1}^j=0$ and  $\varphi_{\mu}(k)_j=-$ if $\mu_{2k+1}^j=a$. Then  $v$ is a $\varphi_{\mu}$-highest vector and $N\simeq L_{\varphi_{\mu}}(\tcH, \lambda)$ up to a shift of gradation. Part(c) is clear. Part(d): without loss of generality we assume that  $\nu_{2k+1}^j=0$ for all $k\in \Z_{\geq 0}$, $j=1, \ldots, n-1$. Clearly, $M_{\varphi_{\nu}}(\tcH, \lambda)$ has finite dimensional graded components. Suppose that $\varphi_{\mu}$ differs   from $\varphi_{\nu}$ only in $s$ places. Consider a nonzero $\varphi_{\mu}$-highest vector  $v$. If $w=x_{2k+1}^jv\neq 0$ for some $k\geq 0$ and $j=1, \ldots, n-1$, then $x_{2k+1}^j w =0$ and thus $w$ is a $\varphi_{\mu'}$-highest vector where $\varphi_{\mu'}$ differs from $\varphi_{\nu}$  in $s-1$ places. Continuing we find  a $\varphi_{\nu}$-highest vector in $M_{\varphi_{\mu}}(\tcH, \lambda)$. Since $M_{\varphi_{\mu}}(\tcH, \lambda)$ is irreducible when $a\neq 0$ we conclude that $M_{\varphi_{\mu}}(\tcH, \lambda)\simeq M_{\varphi_{\nu}}(\tcH, \lambda)$  and hence it has  finite dimensional graded components. Conversely, assume that $M_{\varphi_{\mu}}(\tcH, \lambda)$ has finite dimensional graded components and let $v$ be  a nonzero $\varphi_{\mu}$-highest vector.  Denote by $\Omega_{\mu}$ the subset of odd integers defined as follows: $k\in \Omega_{\mu}$ if $x_{k}^j v\neq 0$ for at least one $j=1, \ldots, n-1$.  A sequence $(k_1, \ldots, k_r)$ of  $\Omega_{\mu}$ is called \emph{cycle} if $\sum_{i=1}^r k_i=0$. Suppose $\Omega$ contains infinitely many positive as well as negative odd integers. Then one can form infinitely many cycles. Each such cycle $(k_1, \ldots, k_r)$ yields a basis element $\Pi_{i=1}^r  x_{k_i}^{j_i} v$ of $M_{\varphi_{\mu}}(\tcH, \lambda)$ which is a contradiction. Hence, $\Omega$ contains only finitely many positive or only finitely negative odd integers. This means that  $\varphi_{\mu}$ differs from $\varphi_{\nu}$  only in finitely many places, where $\nu_{2k+1}^j=0$ for all $k\in \Z_{\geq 0}$, $j=1, \ldots, n-1$, or $\nu_{2k+1}^j\neq 0$ for all $k\in \Z_{\geq 0}$, $j=1, \ldots, n-1$.
\end{proof}

\begin{rem}
We conjecture that any irreducible $\Z$-graded $\tcH$-module is diagonal. 

\end{rem}

We also have the following isomorphism criterion.

\begin{prop}
We have that $M_{\varphi_{\mu}}(\tcH, a)\simeq M_{\varphi_{\mu'}}(\tcH, a')$ (up to a shift of gradation) if and only if $a=a'$ and $\varphi_{\mu}$ and $\varphi_{\mu'}$ differ only in finitely many places.
\end{prop}

\begin{proof}The condition $a=a'$ is clear. Assume that  for some $r$ and $j$,  $d_{2r+1}^j$ has an  eigenvector $v\in M_{\varphi_{\mu}}(\tcH, a)$  with eigenvalue $\mu_{2r+1}^j=a$. Set $w=x_{2r+1}^jv\neq 0$. Then $x_{2r+1}^jw=0$ and hence $w$ is a $\varphi_{\nu}$-highest vector where $\nu_{2k +1}^i=\mu_{2k +1}^i$ if $k\neq  r$ or $i\neq j$, while  $\nu_{2r +1}^j=0$. We have $M_{\varphi_{\nu}}(\tcH, a)\simeq M_{\varphi_{\mu}}(\tcH, a)$.  Similarly, we can change finitely many nonzeros $\mu$'s to zeros. 

Conversely, if we have the isomorphism, then one can obtain a $\varphi_{\mu'}$-highest weight vector by finitely many actions of elements $x_{\pm(2r+1)}^j$ on a $\varphi_{\mu}$-highest weight vector. This implies the statement.
\end{proof}

\section{Irreducibility of generalized Verma type modules}
In this section we prove our main result which is the following theorem.

\begin{theo}\label{thm-extmain}
\begin{enumerate}
\item \label{item1-thm-extmain} $M(\hm, \hk; L(\hk, \lambda))$ and $M(\hs, \lambda)$ are reducible for any $\hs\in \{\hg, \hm, \hk, \cH\}$.
\item \label{item2-thm-extmain} If $\det D_\delta^X(\lambda)\neq 0$, then there is a bijection between submodules of $M(\hm, \hk; L(\hk, \lambda))$ and ideals of $\cS$.
\item \label{item3-thm-extmain} If $\det D_\delta^X(\lambda)\neq 0$, then $M(\hg, \hm; L(\hm, \lambda))$ is irreducible.
\end{enumerate} 
\end{theo}

The next two results imply Theorem~\ref{thm-extmain} items \eqref{item1-thm-extmain} and \eqref{item2-thm-extmain}.

\begin{cor}
$M(\hs, \lambda)$ is reducible for any $\hs\in \{\hg, \hm, \hk, \cH\}$.
\end{cor}
\begin{proof}
It follows from Proposition~\ref{prop:ideals.S.bij.subm.M(H)}.
\end{proof}

\begin{prop}\label{prop:bij.ideals.Verma-hm}
Let $M = M(\hm, \hk; L(\hk, \lambda))$, $L=L(\hk, \lambda)$, and assume that $\det D_\delta^X(\lambda)\neq 0$. Then there is a bijection between submodules of $M$ and ideals of $\cS$; $\cS^+M$ is a maximal proper submodule of $M$; and $L(\hm, \hk; L)\cong \Lambda(\cH_{1}^-)\otimes_\C L$ as vector spaces.
\end{prop}
\begin{proof}
Let $J$ be an ideal of $\cS$. Since $[\cH_0, \hm]=0$, it is clear that $JM$ defines a submodule of $M$. In the other direction, for a submodule $N\subseteq M$, we consider the ideal $J_N\subseteq \cS$ such that $N\cap \cS L = J_N L$. We claim that $J_N = \{a\in \cS\mid av_\lambda\in N\}$. Indeed, let $a\in J_N$, and write an arbitrary $v\in L$ as $uv_\lambda$ for some $u\in \bU(\hk^-)$. Then we have $av = auv_\lambda = uav_\lambda\in N$, and hence $J_NL\subseteq N\cap \cS L$. For the other inclusion, write a general element $v=\sum_{i=1}^m a_iv_i\in N\cap \cS L$ with $a_i\in \cS$ and assume that $v_1,\ldots, v_m\in L$ are linearly independent. The fact that $[\cH_0, \hk]=0$ along with the fact that $L$ is a simple $\hk$-module with countable dimension allow us to apply Jacobson Density Theorem to find, for each $i=1,\ldots, m$, an element $u_i\in \bU(\hk^+)$ for which $u_iv=a_iv_\lambda\in N\cap \cS L$. In particular, $a_i\in J_N$ for every $i$, and the claim is proved.

Now we claim that $N = J_NM$. Indeed, let $jm\in J_NM$. Then, writing $m = u l$ with $u\in \bU(\cH^-)$ and $l\in L$, we get that
	\[
jm=jul=ujl\in uN\subseteq N.	
	\]
Thus $J_NM\subseteq N$. For the other inclusion, consider the canonical projection $\pi: M\to V:=M/J_NM$, $W = \pi(N)$, and $R = \cS/J_N$. Notice that $V$ is free as an $R$-module, and that $W\cap RL=\pi(N\cap \cS L)=\pi(J_N L)=0$. Now if we suppose that $W\neq 0$, then we can use the fact that $\det D_\delta^X(\lambda)\neq 0$, and that $[\cH_0, \hm]=0$, to get a contradiction just as in the proof of Proposition~\ref{prop:ideals.S.bij.subm.M(H)}. Thus $W=0$ and the proof is complete.
\end{proof}

\begin{cor}\label{cor:L(m,lambda)-free-U(H_1^-)}
If $\det D_\delta^X(\lambda)\neq 0$, then $L(\hm, \lambda)\cong \Lambda(\cH_{1}^-)\otimes_\C L(\hk, \lambda)$ as vector spaces.
\end{cor}
\begin{proof}
This follows from $L(\hm, \lambda)\cong L(\hm, \hk; L(\hk, \lambda))$ and Proposition~\ref{prop:bij.ideals.Verma-hm}.
\end{proof}

\begin{cor}
Suppose that $\det D_\delta^X(\lambda)\neq 0$.  Then we have the character formula 
	\[
\ch \ L(\hm, \lambda) = e^\lambda \prod_{\alpha\in \D(X)_{{\rm re},0}^+}(1-e^{-\alpha})^{-1} \prod_{\alpha\in \D(X)_{1}^+}(1+e^{-\alpha})   \prod_{\alpha\in \D(\cH^-)_1}(1+e^{-\alpha}),
	\]
where $\D(X)_{{\rm re},0}^+$ denotes the set of real positive even roots of $\hk$.
\end{cor}
\begin{proof}
This follows from \cite{GS08} along with the fact that $L(\hm, \lambda)\cong \Lambda(\cH_{1}^-)\otimes_\C L(\hk, \lambda)$.
\end{proof}

From now on we assume that 	
	\[
\det D_\delta^X(\lambda)\neq 0, \text{ and hence, by Corollary~\ref{cor:L(m,lambda)-free-U(H_1^-)}, } L(\hm, \lambda)\cong \Lambda(\cH_{1}^-)\otimes_\C L(\hk, \lambda) \text{  as vector spaces. }
	\]

Before proving the irreducibility of $M(\hg, \hm; L(\hm, \lambda))$, we introduce an ordered basis of $M(\hg, \hm; L(\hm, \lambda))$. Recall that for a subalgebra $\fa\subseteq \hg$ we defined $\D(\fa) = \{\alpha\in \D\mid \hg_\alpha\subseteq \fa\}$. Let $B(\fu^-) = \{f_i\in \g_{\alpha_i}\mid \alpha_i\in \dot{\D}(\fu^-)\}$ be a basis of $\fu^-$ such that
	\[
f_i<f_j\text{ if } \alpha_i<\alpha_j.
	\]
Now we order the basis $B(L(\fu^-))=\{f_i(m)\mid m\in \Z\}$ of $L(\fu^-)$ so that
\begin{enumerate}
\item if $m$ is odd and $n$ is even, then $f_i(m)<f_j(n)$,
\item if $m,n$ are both even or both odd, then $f_i(m)<f_j(n)$ if $m<n$, or $m=n$ and $f_i<f_j$.
\end{enumerate}

For $r\geq 1$ and $(\bi, 2\bm, \bp) = (i_1,\ldots, i_r, m_1,\ldots, m_r, p_1,\ldots, p_r)\in \Z_t^r \times 2\Z^r\times \Z_{\geq 0}^r$, we set $f_{\bi, 2\bm, \bp} := f_{i_1}(m_1)^{p_1}\cdots f_{i_r}(m_r)^{p_r}\in \bU(L(\fu^-)_0)$ and we define $\deg f_{\bi, 2\bm, \bp} :=\sum p_i$. For monomials of the different degree we let $f_{\bi, 2\bm,\bp} < f_{\bi', 2\bm', \bp'}$ if $\deg f_{\bi, 2\bm,\bp} < \deg f_{\bi', 2\bm', \bp'}$; for monomials of same degree we define $f_{\bi, 2\bm,\bp} < f_{\bi', 2\bm', \bp'}$ if $({\bi, 2\bm, \bp}) < ({\bi', 2\bm',\bp'})$, where the latter order is the reverse lexicographical order. This provides us a totally ordered basis $B(\bU(L(\fu^-)_0)) = \{ f_{\bi, 2\bm, \bp} := f_{i_1}(m_1)^{p_1}\cdots f_{i_r}(m_r)^{p_r}\}$ of $\bU(L(\fu^-)_0)$. For $r\geq 1$ and $(\bi, \bm) = (i_1,\ldots, i_r, m_1,\ldots, m_r)\in \Z_t^r \times (2\Z^r + 1)$, we set $f_{\bi, \bm} := f_{i_1}(m_1)\cdots f_{i_r}(m_r)$ and we define $\deg f_{\bi, \bm} := r$. For monomials of the different degree we let $f_{\bi, \bm} < f_{\bi', \bm'}$ if $\deg f_{\bi, \bm} < \deg f_{\bi', \bm'}$; for monomials of same degree we define $f_{\bi, \bm} < f_{\bi', \bm'}$ if $({\bi, \bm}) < ({\bi', \bm'})$, where the latter order is the reverse lexicographical order. Finally, we let $ f_{\bi', \bm'}<f_{\bi, 2\bm, \bp} $ for all such monomials. By PBW Theorem, we have that $B(\bU(L(\fu^-))) = \{ f_{\bi, 2\bm, \bp} f_{\bi', \bm'}\}$ is a totally ordered basis of $\bU(L(\fu^-))$. 

Let $h_1,\ldots, h_t$ be a basis of $\fh^X$. Then $H_{i,m} := h_i(-m)$ for $i=1,\ldots, t$ and $m\in \{2\Z_{\geq 0}+1\}$ form a basis for $\cH_{1}^-$. In particular, if we let $H_{i,m}\geq H_{k,n}$ if $m\geq n$ or $m=n$ and $i\geq k$, then the monomials $H_{i_1,m_1}\cdots H_{i_s, m_s}$ with $H_{i_1,m_1}>\cdots > H_{i_s,m_s}$ form a basis $B(\cH_{1}^-) $ of $\Lambda(\cH_{1}^-)$.

Since we are assuming $\det D_\delta^X(\lambda)\neq 0$, Corollary~\ref{cor:L(m,lambda)-free-U(H_1^-)} implies that $L(\hm, \lambda)\cong \Lambda(\cH_{1}^-)\otimes_\C L(\hk, \lambda)$ as vector spaces. Let $\{v_i\mid i\in I\}$ be an ordered basis of $L(\hm, \lambda)$, where the order is induced by the order of $\Lambda(\cH_{1}^-)$. We say $f_{\bi, \bm, \bp} f_{\bi', \bm'} v_i < f_{\bi_1, \bm_1, \bp_1}f_{\bi_1', \bm_1'}v_j$ if $f_{\bi, \bm, \bp} f_{\bi', \bm'} < f_{\bi_1, \bm_1, \bp_1}f_{\bi_1', \bm_1'}$ or if $f_{\bi, \bm, \bp} f_{\bi', \bm'} = f_{\bi_1, \bm_1, \bp_1}f_{\bi_1', \bm_1'}$ and $i < j$. Finally, for an element
	\[
u = \sum u_{\bi,\bm,\bp}^j f_{\bi, \bm, \bp} f_{\bi', \bm'} v_j, \text{ with } u_{\bi,\bm,\bp}^j\in \C,
	\]
we define
	\[
\LinSpan(u) := \Span\{f_{\bi, \bm, \bp} f_{\bi', \bm'}\mid u_{\bi,\bm,\bp}^j\neq 0\}.
	\]

For the next result recall that $L(\hm, \lambda)\cong \Lambda(\cH_{1}^-)\otimes_\C L(\hk, \lambda)$ as vector spaces when $\det (D_\delta^X(\lambda))\neq 0$. Also recall that for $\alpha_i\in \dot{\D}$ we have a triple $f_i\in \g^{-\alpha_i}$, $e_i\in \g^{\alpha_i}$, $h_i\in \fh$ such that $[f_i, e_i]_0 = h_i$.

\begin{lem}\label{lem:comm.rel.basis}
Let ${\bar f} = {\bar f}_0 {\bar f}_1 = f_{i_1}(m_1)^{p_1}\cdots f_{i_r}(m_r)^{p_r} f_{i_1'}(m_1')\cdots f_{i_{r'}'}(m_{r'}')\in B(\bU(L(\fu^-)))$, $v\in L(\hm, \lambda)$ be a nonzero vector, and assume that all factors occurring in ${\bar f}$ are simple. For any such factor $f_{i_l}$, we consider $e_{i_l}\in \fn^+ = \fm^+\oplus \fu^+$. If $\det D_\delta^X(\lambda)\neq 0$, then the following hold:
\begin{enumerate}
\item If $\deg {\bar f}_1 = 0$, then there is $0\gg m_l\in \{2\Z+1\}$ or $0\ll m\in \{2\Z+1\}$ such that
\begin{multline*}
e_{i_l}(m){\bar f} v \\
\equiv \sum_{\substack{1\leq j\leq r \\ i_j = i_l}}^r -p_j(p_j-1) f_{i_j}(m+2m_j){\bar f}^{{\hat j}{\hat j}}v + \sum_{\substack{1\leq j\leq r \\ i_j = i_l}}^r \sum_{\xi =j+1}^{r} p_jp_\xi \alpha_{i_\xi}(h_{i_l}) f_{i_\xi}(m_j + m_\xi + m){\bar f}^{{\hat j}{\hat \xi}}v \\
+ \sum_{\substack{1\leq j\leq r \\ i_j = i_l}}^r p_j {\bar f}^{\hat j} h_{i_l}(m+m_j)v  \mod \bU(L(\fu^-))_{(p-2)}\otimes L(\hm, \lambda).
\end{multline*}
\item If $\deg {\bar f}_1 \geq 1$, then there is $0\gg m\in 2\Z$ or $0\ll m\in 2\Z$ such that
\begin{multline*}
e_{i_l}(m){\bar f} v \\
\equiv \biggl( \sum_{\substack{1\leq j\leq r \\ i_j = i_l}}^r -p_j(p_j-1) f_{i_j}(m+2m_j){\bar f}^{{\hat j}{\hat j}}v + \sum_{\substack{1\leq j\leq r \\ i_j = i_l}}^r \sum_{\xi =j+1}^{r} p_jp_\xi \alpha_{i_\xi}(h_{i_l}) f_{i_\xi}(m_j + m_\xi + m){\bar f}^{{\hat j}{\hat \xi}}v \biggr) \\  
+ \biggl( \sum_{\substack{1\leq j\leq r' \\ i_j = i_l}}^{r'} \sum_{\xi =j+1}^{r'} (-1)^{\xi - (j+1)}{\bar \alpha}_{i_\xi}(h_{i_l}){\bar f}_0 f_{i_\xi}(m_j' + m_\xi' + m) {\bar f}_1^{{\hat j}{\hat \xi}} v \\
+ \sum_{\substack{1\leq j\leq r' \\ i_j = i_l}}^{r'} (-1)^{(r'-j)} {\bar f}^{\hat j} h_{i_l}(m+m_j')v \biggr)  \mod \bU(L(\fu^-))_{(p+r'-2)}\otimes L(\hm, \lambda).
\end{multline*}
\end{enumerate}
\end{lem}
\begin{proof}
We prove part (b) first, as part (a) follows from it. Choose $0\gg m\in 2\Z$ or $0\ll m\in 2\Z$ such that $h_{i_l}(m+m_j')$ is in $B(\cH_{1}^-)$.  Since $\ad (e_{i_l}(m))$ is a derivation of even degree, we have that
\begin{multline*}
e_{a_l}(m){\bar f} v = \sum_{j=1}^r \sum_{\gamma =0}^{p_j-1} f_{i_1}(m_1)^{p_1}\cdots f_{i_j}(m_j)^{\gamma}[e_{i_l}, f_{i_j}]_0(m+m_j)f_{i_j}(m_j)^{p_j-\gamma-1}\cdots f_{i_r}(m_r)^{p_r}{\bar f}_1v \\
+ \sum_{j=1}^{r'} {\bar f}_0 f_{i_1'}(m_1')\cdots f_{i_{j-1}'}(m_{j-1}')[e_{i_l}, f_{i_j'}]_0(m+m_j')f_{i_{j+1}'}(m_{j+1}') \cdots f_{i_{r'}'}(m_{r'}')v \\
= \sum_{\substack{1\leq j\leq r \\ i_j = i_l}}^r \sum_{\gamma =0}^{p_j-1} f_{i_1}(m_1)^{p_1}\cdots f_{i_j}(m_j)^{\gamma}h_{i_j}(m+m_j)f_{i_j}(m_j)^{p_j-\gamma-1}\cdots f_{i_r}(m_r)^{p_r}{\bar f}_1v \\
+ \sum_{\substack{1\leq j\leq r' \\ i_j = i_l}}^{r'} {\bar f}_0 f_{i_1'}(m_1')\cdots f_{i_{j-1}'}(m_{j-1}')h_{i_j}(m+m_j')f_{i_{j+1}'}(m_{j+1}') \cdots f_{i_{r'}'}(m_{r'}')v \\
\equiv \biggl( \sum_{\substack{1\leq j\leq r \\ i_j = i_l}}^r -p_j(p_j-1) f_{i_j}(m+2m_j){\bar f}^{{\hat j}{\hat j}}v + \sum_{\substack{1\leq j\leq r \\ i_j = i_l}}^r \sum_{\xi =j+1}^{r} p_jp_\xi \alpha_{i_\xi}(h_{i_l}) f_{i_\xi}(m_j + m_\xi + m){\bar f}^{{\hat j}{\hat \xi}}v \\
+ \sum_{\substack{1\leq j\leq r \\ i_j = i_l}}^r p_j {\bar f}^{\hat j} h_{i_l}(m+m_j)v \biggr)  + \biggl( \sum_{\substack{1\leq j\leq r' \\ i_j = i_l}}^{r'} \sum_{\xi =j+1}^{r'} (-1)^{\xi - (j+1)}{\bar \alpha}_{i_\xi}(h_{i_l}){\bar f}_0 f_{i_\xi}(m_j' + m_\xi' + m) {\bar f}_1^{{\hat j}{\hat \xi}} v +\\
 \sum_{\substack{1\leq j\leq r' \\ i_j = i_l}}^{r'} (-1)^{(r'-j)} {\bar f}^{\hat j} h_{i_l}(m+m_j')v \biggr)  \mod \bU(L(\fu^-))_{(p+r'-2)}\otimes L(\hm, \lambda) \\
 \equiv \biggl( \sum_{\substack{1\leq j\leq r \\ i_j = i_l}}^r -p_j(p_j-1) f_{i_j}(m+2m_j){\bar f}^{{\hat j}{\hat j}}v + \sum_{\substack{1\leq j\leq r \\ i_j = i_l}}^r \sum_{\xi =j+1}^{r} p_jp_\xi \alpha_{i_\xi}(h_{i_l}) f_{i_\xi}(m_j + m_\xi + m){\bar f}^{{\hat j}{\hat \xi}}v \biggr) \\  
+ \biggl( \sum_{\substack{1\leq j\leq r' \\ i_j = i_l}}^{r'} \sum_{\xi =j+1}^{r'} (-1)^{\xi - (j+1)}{\bar \alpha}_{i_\xi}(h_{i_l}){\bar f}_0 f_{i_\xi}(m_j' + m_\xi' + m) {\bar f}_1^{{\hat j}{\hat \xi}} v \\
+ \sum_{\substack{1\leq j\leq r' \\ i_j = i_l}}^{r'} (-1)^{(r'-j)} {\bar f}^{\hat j} h_{i_l}(m+m_j')v \biggr)  \mod \bU(L(\fu^-))_{(p+r'-2)}\otimes L(\hm, \lambda),
\end{multline*}
where the first equivalence follows from the fact that $\ad (h_{i_l}(m+m_j))$ is an even derivation, $\ad(h_{i_j}(m+m_j'))$ is an odd derivation, and $f_{i_\xi'}(m_\xi')$ is an odd element for any $m_\xi'$; and the second equivalence follows from the fact that $h_{i_l}(m+m_j)v =0$ for all $1\leq j\leq r$, since either $h_{i_l}(m+m_j)\in \cH_{0}^+$ that implies $h_{i_l}(m+m_j)v =0$, or $h_{i_l}(m+m_j)\in \cS^+$, and hence $h_{i_l}(m+m_j)v$ lies in the maximal proper submodule of $M(\hk, \lambda)$.

For part (a), we notice that the second parentheses above does not appear in the expression of $e_{a_l}(m){\bar f} v$. Moreover, despite the fact that $\ad(e_{i_l}(m))$ and $\ad(h_{i_l}(m+m_k))$ are odd derivations (as $m\in \{2\Z+1\}$ and $m_j\in 2\Z$ for all $1\leq j\leq r$), they behave as regular derivations when applied on factors of ${\bar f}_0$, since $m_j\in 2\Z$ for all $1\leq j\leq r$. Thus the proof follows from the above equation.
\end{proof}

We now state our key result.

\begin{theo}\label{thm-main}
If $\det D_\delta^X(\lambda)\neq 0$, then $M(\hg, \hm; L(\hm, \lambda))$ is irreducible.
\end{theo}
\begin{proof}
We claim that any non-trivial submodule $N$ of $M(\hg, \hm; L(\hm, \lambda))$ intersects $L(\hm, \lambda)$ non-trivially. Assuming that the claim holds, the result follows from the simplicity of $L(\hm, \lambda)$. 

To prove the claim, let $0\neq v\in N_\mu$, and let ${\bar f}_{\max} x_{\max} = f_{\ba, 2\bb, \bc} f_{\ba', \bb'} x_d$ be  the maximal monomial occurring in $v$. We now reduce the proof to the case where all  factors $f_{i_j}$ of maximal degree monomials occurring in $v$ are simple root vectors. Indeed, consider all factors $f_{i_j}$ that occur in the monomials of maximal degree of $v$, and let $f_{i_k}$ be the minimal among them (i.e. its associated root $\alpha_{i_k}$ is such that $|\alpha_{i_k}|$ is maximal among them). Let ${\bar f}_{\min}x_{\min} = f_{\bd, 2\bg, \bk} f_{\bd', \bg'} h_{\bd'', \bg''}x_{\min} = {\bar f}_{0,\min} {\bar f}_{1,\min}x_{\min}$ be an element (occurring in $v$) of maximal degree having $f_{i_k}$ as a factor, and let $z\in \fn^+ = \fm^+\oplus \fu^+$ be such that $0\neq [z, f_{i_k}]\in \fu^-$ (such $z$ exists by \cite[Lemma~4.2]{Cox94}). Let $J_{\min}$ the set of indexes $j$ for which $f_{i_j}$ is a factor of ${\bar f}_{\min}$ and $[z, f_{i_j}]\in \fu^-$. Let $0\gg m\in 2\Z$ (if $z\in \fu^+$) or $0\ll m\in 2\Z$ (if $z\in \fm^+$) (here $m\ll 0$ (resp. $m\gg 0$) means $m$ so that for every fixed $j$, $m+m_j\notin \{g_k,\ g_l'\mid 1\leq k\leq r,\ 1\leq l\leq r'\}$). Then, using that $\ad(z(m))$ is an even derivation, we obtain
\begin{multline*}
z(m){\bar f}_{\min} x_{\min} =\\
= \sum_{j=1}^r \sum_{\gamma =0}^{k_j-1} f_{d_1}(g_1)^{k_1}\cdots f_{d_j}(g_j)^{\gamma}[z, f_{d_j}]_0(m+g_j)f_{d_j}(g_j)^{k_j-\gamma-1}\cdots f_{d_r}(g_r)^{k_r}{\bar f}_{1,\min}x_{\min} \\
+ \sum_{j=1}^{r'} {\bar f}_{0,\min} f_{d_1'}(g_1')\cdots f_{d_{j-1}'}(g_{j-1}')[z, f_{d_j'}]_0(m+g_j')f_{d_{j+1}'}(g_{j+1}') \cdots f_{d_{r'}'}(g_{r'}')x_{\min} \\
 \equiv \sum_{j\in J_-} k_j [z, f_{i_j}]_0(m+g_j){\bar f}_{\min}^{\hat j}x_{\min} \\
+ \sum_{j\in J_-} (-1)^{j-1}{\bar f}_{0,\min}[z, f_{i_j}]_0(m+g_j'){\bar f}_{1,\min}^{\hat j}x_{\min}  \mod \bU(L(\fu^-))_{(k+d'-1)}\otimes L(\hm, \lambda),
\end{multline*}
where $k+d'= \deg{\bar f}_{\min}$. Now if $S_1$ denote this summation, then it is nonzero since $[z, f_{i_k}]_0\neq 0$ and $m+m_j\notin \{g_k,\ g_l'\mid 1\leq k\leq r,\ 1\leq l\leq r'\}$. Moreover, if ${\bar f} x = {\bar f}_0{\bar f}_1x$ is a different monomial occurring in $v$, then, similarly we have that
\begin{multline*}
z(m){\bar f}x \equiv \sum_{j\in J_-} p_j [z, f_{i_j}]_0(m+m_j){\bar f}^{\hat j}x \\
+ \sum_{j\in J_-} (-1)^{j-1}{\bar f}_{0}[z, f_{i_j}]_0(m+m_j'){\bar f}_{1}^{\hat j}x  \mod \bU(L(\fu^-))_{(p+r-1)}\otimes L(\hm, \lambda),
\end{multline*}
where $p + r = \deg {\bar f}$. Since ${\bar f}_{\min}$ has maximal degree among  monomials in $v$, we have that $p+r\leq k + d'$. Hence, if $T_1$ is the summation above, then $S_1\notin \LinSpan (T_1) + \bU(L(\fu^-))_{(p-1)}\otimes L(\hm, \lambda)$, since this could happen only if $p+r=k+d'$; $\C[z,f_{i_j}]_\ell = \C[z, f_{i_l}]_\ell$ for $\ell = 0,1$; $m_j = g_l$; ${\bar f}_{\min}^{\hat j} = {\bar f}^{\hat l}$; and $x_{\min}=x$. But this would imply ${\bar f}_{\min} = {\bar f}$, and $x_{\min}=x$, which contradicts the fact that ${\bar f} x\neq {\bar f}_{\min} x_{\min}$.

We may now assume that factors of all maximal degree monomials occurring in $v$ are simple. In particular, this is the case for 
	\[
{\bar f}_{\max} x_{\max} =  {\bar f}_{0,\max} {\bar f}_{1, \max} x_{\max} =  f_{a_1}(b_1)^{c_1}\cdots f_{a_s}(b_s)^{c_s} f_{a_1'}(b_1')\cdots f_{a_{s'}}(b_{s'})x_{\max}.
	\] 
Moreover, we may also assume that $\deg f_{1,\max}\geq 1$ (as otherwise the proof is the same as that of \cite[Proposition~4.5]{Cox94}, using Lemma~\ref{lem:comm.rel.basis} and a suitable $e\in \{2\Z + 1\}$ in his notation). By Lemma~\ref{lem:comm.rel.basis}, for each simple root factor $f_{a_l}$ of ${\bar f}_{\max}$, there is $0\gg m\in 2\Z$ or $0\ll m\in 2\Z$ for which 
\begin{multline*}
e_{a_l}(m){\bar f}_{\max}x_{\max}\equiv \biggl( \sum_{\substack{1\leq j\leq s \\ i_j = i_l}}^s -c_j(c_j-1) f_{a_j}(m+2b_j){\bar f}_{\max}^{{\hat j}{\hat j}}x_{\max}  \\
+ \sum_{\substack{1\leq j\leq s \\ i_j = i_l}}^r \sum_{\xi =j+1}^{r} c_jc_\xi \alpha_{a_\xi}(h_{a_l}) f_{a_\xi}(b_j + b_\xi + m){\bar f}_{\max}^{{\hat j}{\hat \xi}}x_{\max}  \biggr)  \\
+ \biggl( \sum_{\substack{1\leq j\leq s' \\ i_j = i_l}}^{s'} \sum_{\xi =j+1}^{s'} (-1)^{\xi - (j+1)}{\bar \alpha}_{a_\xi}(h_{a_l}){\bar f}_{0, \max} f_{a_\xi}(b_j' + b_\xi' + m) {\bar f}_{1,\max}^{{\hat j}{\hat \xi}} x_{\max} \\
+ \sum_{\substack{1\leq j\leq s' \\ i_j = i_l}}^{s'} (-1)^{(s'-j)} {\bar f}_{\max}^{\hat j} h_{i_l}(m+b_j')x_{\max} \biggr) \mod \bU(L(\fu^-))_{(c+s'-2)}\otimes L(\hm, \lambda)
\end{multline*}
Finally, it is not hard to prove that for any fixed index $l$, the summand 
	\[
w_l = {\bar f}_{\max}^{\hat l}h_l(m+b_l') x_{\max}
	\] 
is not in the $\LinSpan$ of the remaining monomials occurring in $e_l(m) v$. Therefore, $e_l(m)v\neq 0$, the maximal monomial occurring in $e_l(m)v$ has degree less than that of the maximal monomial occurring in  $v$, and thus the result follows by induction.
\end{proof}

Applying Theorem \ref{thm-main} in the case $X=\emptyset$ gives:

\begin{cor}\label{cor:X=empty-irr.criterion}
If $\det D_\delta^\emptyset (\lambda)\neq 0$, then $M(\hg, \cH; L(\cH, \lambda))$ is irreducible. 
\end{cor}

\begin{rem}
Notice that differently from the other cases studied in the literature, we do not need the central charge to be nonzero in order to have $M(\hg, \hm; L(\hm, \lambda))$ to be irreducible (compare with \cite{Cox94, Fut94, CF18}). This is due to the fact that the central element $K$ does not play a role in the action of the imaginary subalgebra $\cH$ on $L(\hm, \lambda)$. On the other hand, the condition $\det D_\delta^X(\lambda)\neq 0$ is essential in our context. Without this condition we do not necessarily have that $L(\hm, \lambda)\cong \Lambda(\cH_{1}^-)\otimes_\C L(\hk, \lambda)$ as vector spaces (see Corollary~\ref{cor:L(m,lambda)-free-U(H_1^-)}).
\end{rem}

\medskip

\noindent {{\bf Acknowledgment.} Second author is grateful to the University of California Berkeley for hospitality. The authors are very grateful to Vera Serganova for very stimulating discussions.



\begin{thebibliography}{ASM}

\bibitem[BBFK13]{BBFK13}
V.~Bekkert, G.~Benkart, V.~Futorny, and I.~Kashuba.
{\it New irreducible modules for {H}eisenberg and affine {L}ie algebras},
J. Algebra  {\bf 373} ({2013}), 284--298.

\bibitem[CF18]{CF18}
L.~Calixto and V.~Futorny.
{\it Highest weight modules for affine {L}ie superalgebras},
arXiv:1804.02563 ({2018})

\bibitem[Cox94]{Cox94}
B.~Cox.
{\it Verma modules induced from nonstandard {B}orel subalgebras},
Pacific J. Math. {\bf 165} ({1994}), no. 2, 269--294.

\bibitem[DFG09]{DFG09}
I.~Dimitrov, V.~Futorny, and D.~Grantcharov.
{\it Parabolic sets of roots},
in Groups, rings and group rings; Amer. Math. Soc., Providence, RI, vol. 499, Contemp.
  Math., 2009, 61--73.

\bibitem[ERF09]{EF09}
S.~Eswara~Rao and V.~Futorny.
{\it Integrable modules for affine {L}ie superalgebras},
Trans. Amer. Math. Soc., {\bf 361} ({2009}), no. 10, 5435--5455.

\bibitem[Fut94]{Fut94}
V.~Futorny.
{\it Imaginary {V}erma modules for affine {L}ie algebras},
Canad. Math. Bull., {\bf 37} ({1994}), no. 2, 213--218.

\bibitem[Fut97]{Fut97}
V.~Futorny.
{\it Representations of affine {L}ie algebras}, 
Queen's Papers in Pure and Appl. Math., vol. 106, Queen's University, Kingston, ON ({1997}).


\bibitem[GS08]{GS08}
M.~Gorelik and V.~Serganova.
{\it On representations of the affine superalgebra
  {${\mathfrak{q}(n)^{(2)}}$}},
Mosc. Math. J., {\bf 8}  ({2008}), no. 1, 91--109, 184.

\bibitem[HS07]{HS07}
C.~Hoyt and V.~Serganova.
{\it Classification of finite-growth general {K}ac-{M}oody superalgebras},
Comm. Algebra, {\bf 35} {(2007)}, no. 3, 851--874.

\bibitem[Kac77]{Kac77}
V.~Kac.
{\it {L}ie superalgebras},
Advances in Math., {\bf 26} {(1977)}, no. 1, 8--96.

\bibitem[Ser11]{Ser11}
V.~Serganova.
{\it Kac-{M}oody superalgebras and integrability},
in Developments and trends in infinite-dimensional {L}ie
  theory, Birkh\"auser Boston, Inc., Boston, MA, vol. 288,  Progr. Math., 2011, 169--218.

\bibitem[vdL89]{Van89}
Johan~W. van~de Leur.
{\it A classification of contragredient {L}ie superalgebras of finite growth},
Comm. Algebra, {\bf 17} ({1989}), no. 8, 1815--1841.

\end{thebibliography}
\end{document}